\documentclass{amsart}
\usepackage[utf8]{inputenc}
\usepackage[margin=1in]{geometry}
\usepackage{amsmath,amsthm,amssymb,amsfonts}
\usepackage{xcolor}
\usepackage{graphicx}
\usepackage{subcaption}
\usepackage{xfrac}
\usepackage{mathabx}
\usepackage{physics}
\usepackage{enumitem}
\usepackage{cleveref}
\usepackage{bbm}
\usepackage{epstopdf}
\usepackage{algorithmic}
\usepackage{caption}
\usepackage{ulem}

\newtheorem{theorem}{Theorem}[section]
\newtheorem{lemma}{Lemma}[section]
\newtheorem{corollary}{Corollary}[section]
\newtheorem{remark}{Remark}[section]
\newtheorem{definition}{Definition}[section]
\newtheorem{proposition}{Proposition}[section]

\begin{document}
% Title and author information
\title[Speed of propagation of fractional dispersive waves]{Speed of propagation of fractional dispersive waves}

\author{Brian Choi}
\address{Department of Mathematics, Southern Methodist University, Dallas, TX 75275, USA}
\email{choigh@bu.edu}

\author{Steven Walton}
\address{Theoretical Division, Los Alamos National Laboratory, Los Alamos NM 87545, USA}
\email{stevenw@lanl.gov}

\thanks{The authors were supported in part by [NSF RTG grant DMS-1840260] at Southern Methodist University. The second author is with Los Alamos National Laboratory; this manuscript is approved for unlimited release (LA-UR-24-21704).}
\thanks{Corresponding author: Brian Choi}

% Abstract
\begin{abstract}
In this paper, we show that all non-trivial solutions of a broad class of nonlinear dispersive equations, whose linear evolution is governed by a dispersion relation under minimal regularity assumptions, cannot remain compactly supported for any non-trivial time interval. Our approach, based on complex-analytic arguments and the Paley-Wiener-Schwartz theorem, yields a stronger result: if linear solutions are compactly supported at two distinct times, then the dispersion relation must admit an analytic extension. This extends previous results beyond polynomial dispersion relations and applies to more general settings, including fractional-order systems. As an application, we examine the generalized space-time fractional Schr\"{o}dinger equation, illustrating the role of memory effects in wave propagation.
\end{abstract}

% Keywords
\keywords{Fractional dispersive equation, Paley-Wiener-Schwartz theorem, Infinite speed of propagation, Unique continuation\\
\textbf{MSC 2020} 35Q55, 35B60, 35S30, 26A33, 35B40, 35C15}

\maketitle

\section{Introduction}

A PDE exhibits infinite speed of propagation (ISP) when a localized initial datum with compact support instantly spreads throughout space. This phenomenon plays a crucial role in dispersive wave equations, where it relates to unique continuation principles and qualitative properties of solutions. In this paper, we extend previous results on ISP \cite{ZHANG1994290,bourgain1997compactness} beyond polynomial dispersion relations using complex-analytic techniques, particularly the Paley-Wiener-Schwartz theorem and the holomorphic Fourier transform, to analyze solution support for a broad class of nonlinear dispersive equations.

Consider the Cauchy problem
\begin{equation}\label{FSE2}
\begin{split}
i\partial_t u &= Lu, (x,t) \in \mathbb{R}^{d+1},\\
u(x,0) &= \phi(x),    
\end{split}
\end{equation}
where $L$ is a Fourier multiplier given by $\widehat{Lf}(\xi) = \omega(\xi)\widehat{f}(\xi)$; see \Cref{notation} for notations. By a one-parameter family of unitary operators $U(t) = \exp(-itL)$ and the fundamental solution $K_t = U(t)\delta$, the solution has an analytic form
\begin{equation*}
    u(x,t) = U(t)\phi (x) = K_t \ast \phi(x) = \frac{1}{(2\pi)^d} \int e^{i(x\cdot\xi -t \omega(\xi))}\widehat{\phi}(\xi)d\xi.
\end{equation*}

However, direct verification of ISP in the physical space is generally intractable unless $\omega(\xi)$ takes a particularly simple algebraic form, such as the Schr\"{o}dinger dispersion relation $\omega(\xi) = |\xi|^2$ or the Airy-type relation $\omega(\xi) = \xi^3$. Even in these cases, establishing ISP for general initial data via convolution against $K_t$ remains non-trivial. Since a wide class of dispersion relations, including many that are non-analytic, arise in physically relevant models, and such relations often lack convenient mathematical structure, a detailed analysis of their behavior poses additional challenges. Examples include the space-fractional Schr\"{o}dinger equation with $\omega(\xi) = |\xi|^\alpha$ and the nondimensionalized surface gravity wave model where $\omega(\xi)^2 = (\xi + \xi^3) \tanh \xi$. Unlike polynomial dispersion relations, which allow for explicit algebraic manipulations, non-analytic dispersions often require integral representations involving special functions, such as the Fox H-function (see \eqref{H-function}). This makes direct analysis in the physical domain significantly more challenging.

To address this, we develop a method based on the Paley-Wiener-Schwartz theorem to prove that ISP is a generic phenomenon among dispersive equations of type \eqref{FSE2} and its nonlinear generalizations under minimal regularity assumptions on $\omega(\xi)$. Unlike prior work that restricted $\omega$ to polynomials (e.g. \cite{ZHANG1994290}), our result (\Cref{thm}) extends the previous analysis to the class of continuous dispersion relations (the continuity hypothesis is sharp by \Cref{rmk1}) and shows that if the support of the flow given by $U(t)$ is compact at any two distinct times, then $\omega$ extends analytically into the complex plane. On the other hand, if $\omega$ is a polynomial of degree at least two, then a stronger conclusion is shown via the Fourier–Bros–Iagolnitzer (FBI) transform: any non-trivial solution spreads out to the entire spatial domain infinitely fast. When $U(t)$ is applied to $u(x,0) = \delta(x)$, the delta function, \Cref{thm} yields the statement (see also \Cref{rmk_liouville}):
\begin{proposition}\label{fundamental_cpt}
Let $\omega \in C(\mathbb{R}^d;\mathbb{R})$. Then $\omega(\xi) = \sum\limits_{|\boldsymbol\alpha|\leq 1}a_{\boldsymbol\alpha }\xi^{\boldsymbol\alpha}$ where $\boldsymbol\alpha \in \mathbb{N}_{0}^{d}$ and $a_{\boldsymbol\alpha } \in \mathbb{R}$ if and only if $supp(K_t)\subseteq \mathbb{R}^d$ is compact for some $t \neq 0$.
\end{proposition}

ISP holds for a certain class of \textit{nonlinear equations}. For the standard and modified KdV equations, \cite{zhang1992unique} showed that the solutions cannot be compactly supported at two distinct times by using the complete integrability of the models. In most cases, a given PDE fails to be completely integrable, however. This issue was addressed in \cite{bourgain1997compactness} by showing that the KdV-type equation cannot have a solution compactly supported for all $t \in I$ by treating the nonlinear Duhamel contribution as a small perturbation. We show that the method in \cite{bourgain1997compactness} generalizes by considering $\omega$ that satisfies a certain growth condition and the power-type nonlinearity $N[u] = u^k$ (or any k-fold product for $k \geq 2$ involving complex conjugates) such that $u(x,t)$ is a solution to
\begin{equation}\label{nonlinear_eqn}
    i\partial_t u = Lu + N[u]
\end{equation}
for $t \in I$. We give a simplified version of \Cref{thm_nonlinear}.

\begin{proposition}\label{prop_nonlinear}
    Let $\omega$ be entire on $\mathbb{C}^d$ satisfying the growth condition \eqref{growth}. Suppose $u(x,t)$ is a smooth solution to \eqref{nonlinear_eqn} whose spatial support is bounded in a ball of radius $O(1)$ in the time interval of existence. Then $u \equiv 0$, i.e., a trivial solution.
\end{proposition}

However this growth condition is not sharp as it fails to apply to the power-type nonlinear Schr\"odinger equations whose analysis in the context of unique continuation was done in \cite{kenig2003unique}; see \Cref{rmk:3.1}. For extensions of \cite{bourgain1997compactness} to specific nonlinear dispersive models with $\omega$ growing sufficiently fast, including the Zakharov-Kuznetsov and the KP-II equation, see \cite{panthee2004note,panthee2005unique,carvajal2005unique}. However, as there exist compactly supported soliton solutions (compactons) to highly nonlinear models (see \cite{rosenau1993compactons}), more remains to be studied in nonlinear regimes.

It is worthwhile to mention the method of Carleman estimates in establishing the unique continuation of solutions, though {\it we do not take this approach}. In \cite{kenig1988note}, for $d \geq 1$, if the solution $u$ of the linear Schr\"odinger equation satisfies some integrability condition and satisfies the differential inequality of the classical Schr\"odinger operator $|(i\partial_t + \Delta)u| \leq |Vu|$ for potentials $V \in L^{\frac{d+2}{2}}(\mathbb{R}^{d+1})$ then the authors show that $u$ satisfies a Carleman estimate.  Thus, they are able to show that if $u$ vanishes in a half-space of $\mathbb{R}^{d+1}$, then $u \equiv 0$. Nonlinear, local dispersive models have been considered for generalizations of the Korteweg de Vries equation, e.g. the Zakharov-Kuznetsov equation, and an alternative model, the Benjamin-Bona-Mahony equation, in \cite{kenig2002support, panthee2004note, rosier2013unique}, respectively. The generalized Kawahara equation is considered in \cite{zhang2018well}.  Extensions to nonlinear, non-local models are carried out in \cite{kenig2020unique} and \cite{10.57262/die/1356060510}.  The above brief list is by no means exhaustive and we refer the reader to the references found within each of the citations provided above.  An advantage in using Carleman estimates is that once one has shown the desired Carleman inequality holds then well-known properties of these estimates may be used to obtain unique continuation properties.  The applications of Carleman estimates are quite vast and, in addition to the works mentioned above, have been used to demonstrate controllability results (see also \cite{Lebeau2012, fursikov1996controllability} for example) for a large class of partial differential equations, usually by demonstrating that an observability inequality holds.

In addition to classical dispersive equations, ISP in time-fractional systems presents new challenges due to memory effects and anomalous diffusion. To model such effects, we consider the Caputo fractional derivative, defined as follows:
\begin{equation*}
    \partial_t^\alpha u (t) = \frac{1}{\Gamma(1-\alpha)} \int_0^t \frac{\partial_t u (\tau)}{(t-\tau)^\alpha} d\tau,
\end{equation*}
for $0 < \alpha \leq 1$. Fractional time derivatives are used to model physical phenomena with long-range correlation in time including relaxation processes. This topic has been an area of active research in recent times. For survey articles, see \cite{he2014tutorial} and references therein. Solutions to certain fractional (in time) Cauchy problems can be realized as the scaling limit of continuous time random walks; a famous non-fractional example is an appropriate scaling limit of simple random walks that gives rise to the Brownian motion. See \cite{meerschaert2009fractional} for a more thorough introduction to the topic.

% \textcolor{blue}{1. Explain the fractionalization of i. References: Narahari, Naber. Emphasize the qualitative difference of our dispersive estimates with those of Su. Use the figure.} 

Finally, we explore ISP in time-fractional dispersive systems, particularly the generalized space-time fractional Schrödinger equation,
\begin{equation}\label{stfsch}
    i^\gamma \partial_t^\alpha u = (-\Delta)^{\frac{\beta}{2}}u,\ (x,t) \in \mathbb{R}^{d+1},\ u(x,0) = \phi(x),
\end{equation}
and explore its various properties including the non-existence of compactly supported solutions, the asymptotic convergence to (almost) traveling waves, and the dispersive estimates depending on how the energy operator $i\partial_t$ is fractionalized. In \cite{laskin2018fractional}, the Laplacian was replaced by the fractional Laplacian by re-deriving the model where the action functional is integrated against the $\alpha$-stable L\'evy-like paths instead of Brownian-like paths, preserving the unitarity of time evolution. The controversy of how to fractionalize the imaginary number $i$ was raised in \cite{naber2004time,achar2013time} where the former suggested $\gamma = \alpha$ and the latter suggested $\gamma = 1$. The analysis of dispersive decay when $\gamma = 1$ was worked in \cite{su2021dispersive} where an estimate analogous to \eqref{disp_est3} was derived. Our analysis, on the other hand, provides a novel insight that the decay depends on both the memory effect (manifested by $\alpha$) and the spatial dimension, thereby extending the previous work \cite{su2021dispersive}. More specifically, the full range of dispersive estimates for $\gamma < 1$ is derived; see \Cref{disp_est} and \Cref{fig} that numerically verifies the role of $\alpha,\gamma$ in the time evolution of linear solutions. Although the physical significance of raising $i$ to a power is not discussed in our paper, see \cite{wang2007generalized,dong2008space,lee2020strichartz} and the references therein for the analyses of \eqref{stfsch}.

The paper is organized as follows. In \Cref{analytic}, we state \Cref{thm} along with the proof of analytic regularity of the dispersion relation. The key idea is two-fold. The Fourier transform of a compactly supported distribution is entire, and that a non-trivial entire function cannot decay too rapidly (\Cref{lb entire}). In \Cref{nonlinearization}, we state \Cref{thm_nonlinear}. The growth of $\omega$ is used to control the nonlinear evolution by the linear term. In \Cref{full_domain}, the support of $u(\cdot,t)$ is shown to be the entire Euclidean domain under certain hypotheses. Note that this is a stronger statement than the claim that the support is non-compact. To better understand the microlocal properties of the fundamental solution, the FBI transform is used along with the method of steepest descent to show that $K_t$ is real-analytic in the spatial variable for $t \neq 0$. In \Cref{frac_trav}, \eqref{stfsch} is studied under various fractional regimes, leading to distinct qualitative behaviors in the dispersive decay when $\alpha \neq \beta$. On the other hand, the non-dispersive case $\alpha = \beta < 1$ yields the convergence to the half-wave equation in the long time limit. 

\subsection{Notation.}\label{notation}

Denote $\mathcal{S}$ by the Schwartz class and $\mathcal{S}^\prime$, the space of tempered distributions. Denote $\mathcal{E}$ by the space of smooth functions and $\mathcal{E}^\prime$, the space of compactly supported distributions. The Fourier transform is defined as

\begin{equation*}
    \widehat{f}(\xi) = \mathcal{F}[f](\xi) = \int f(x)e^{-ix\cdot \xi}dx,\ \mathcal{F}^{-1}[f](x) = (2\pi)^{-d} \int f(\xi)e^{ix\cdot\xi}d\xi,  
\end{equation*}
and $i^\alpha := e^{i\frac{\pi \alpha}{2}}$ by setting the negative real axis as a branch cut for the complex logarithm.

Let $\psi \in C^\infty_c(\mathbb{R}^d)$ be a smooth bump function supported in $B(0,2)$ and define $\zeta(\xi) = \psi(\xi) - \psi(2\xi)$. Define $\zeta_N(\xi) = \zeta\left(\frac{\xi}{N}\right)$ and $P_N f = \mathcal{F}^{-1} [\zeta_N \widehat{f}]$. For $p,q \in [1,\infty],\ s \in \mathbb{R}$, define homogeneous and inhomogeneous Besov spaces as
\begin{equation*}
\dot{B}^s_{p,q} = \Bigl\{f \in \mathcal{S}^\prime / \mathcal{P}: \left(\sum_{N \in 2^\mathbb{Z}} N^{sq} \| P_N f \|_{L^p}^q\right)<\infty\Bigl\};\ B^s_{p,q} = \Bigl\{f \in \mathcal{S}^\prime: \left(\sum_{N \in 2^\mathbb{Z}} \langle N \rangle^{sq} \| P_N f \|_{L^p}^q\right)<\infty\Bigl\}
,   
\end{equation*}
respectively, where $\langle z \rangle = (1+|z|^2)^{\frac{1}{2}}$ and $\mathcal{P}$ is the space of polynomials on $\mathbb{R}^d$.

The Mittag-Leffler function is defined by
\begin{equation*}
    E_{\alpha,\beta}(z) = \sum_{k=0}^\infty \frac{z^k}{\Gamma(\alpha k + \beta)},
\end{equation*}
for $\alpha,\beta>0$ and $z \in \mathbb{C}$. For $t>0, \lambda \in \mathbb{C},\ 0<\alpha\leq 1$, it can be directly verified that $y(t) = E_\alpha(\lambda t^\alpha)y(0)$ satisfies
\begin{equation*}
    \partial_t^\alpha y(t) = \lambda y(t).
\end{equation*}

\section{Analytic extension of the dispersion relations}\label{analytic}

Given a function $f(x_1,\dots,x_d)$, for $x\in \mathbb R^d$ define $x^\prime_j:= (x_1, x_2, \ldots, x_{j-1}, 0, x_{j+1}, \ldots,x_d)$ and the map 
\begin{equation}\label{inclusion}
    x_j \mapsto f(x_j;x_j^\prime) := f(x_1,\dots,x_{j-1},x_j,x_{j+1},\dots,x_d)
\end{equation} where $1 \leq j \leq d$ with $x_1,\dots,x_{j-1},x_{j+1},\dots,x_d$ fixed. In the sequel, $x_{j}^{\prime}$ is considered an element of $ \mathbb R^{d-1}$. 

Denote $\xi = (\xi_1,\dots,\xi_d)$ and $z_j = \xi_j + i \eta_j$ where $\xi_j,\eta_j \in \mathbb{R}$. For an open subset $\Omega \subseteq \mathbb{C}^d$, recall that $f:\Omega\rightarrow \mathbb{C}$ is holomorphic if $f$ is holomorphic in each variable. Given $z \in \mathbb{C}\setminus \{\eta=0\}$, define a ray in $\mathbb{C}$ as
\begin{center}
$\gamma_z =
\begin{cases}
    \{\xi = \Re(z),\ \eta \geq \Im(z)\},
    & \Im(z)>0\\
    \{\xi = \Re(z),\ \eta \leq \Im(z)\}, 
    & \Im(z)<0. 
\end{cases}$    
\end{center}

One of the main tools in our proof is the Paley-Wiener-Schwartz Theorem. Recall that $u \in \mathcal{S}^\prime$ has order $N \in \mathbb{N}_0$ if there exists $C>0$ such that
\begin{equation*}
    |u(\phi)| \leq C \sup_{|\boldsymbol\alpha|\leq N} \sup_{x \in \mathbb{R}^d} \langle x \rangle^N  |\partial_{\boldsymbol\alpha} \phi(x)|,
\end{equation*}
for all $\phi \in \mathcal{S}$.  Note that we are using $\boldsymbol\alpha$ to denote a mulit-index which should not be confused with $\alpha$ in $\partial_t^\alpha$ as defined in section \ref{notation}.  

The following proposition provides a basis for the characterization of solutions with compact support. 
\begin{proposition}
[{\cite[Theorem 7.23]{rudin1974functional}}]\label{PWS}
Let $u \in \mathcal{E}^\prime$ of order $N$ be supported in $rB$ where $B\subseteq \mathbb{R}^d$ is the unit ball centered at the origin and $r>0$. Then $f(z) := u(e_z)$, where $e_z(\xi):=e^{i\xi \cdot z},\ \xi\in\mathbb{R}^d$, is an entire extension of the (inverse) Fourier transform of $u$. Furthermore $f$ satisfies
\begin{equation}\label{exponential}
    |f(z)| \leq C \langle z \rangle^{N} e^{r|\Im(z)|},\ \forall z \in \mathbb{C}^d,
\end{equation}
for some $C>0$. Conversely if an entire function $f$ satisfies \eqref{exponential} for some $C>0,\: N \in \mathbb{N}_0,\ r>0$, then $u:= \mathcal{F}[f|_{\mathbb{R}^d}] \in \mathcal{E}^\prime$ is supported in $rB$.\footnote{However $u$ need not be of order $N$.}
\end{proposition}

The following theorem offers an easily verifiable condition characterizing when the solution to the dispersive equation cannot remain compactly supported. 
\begin{theorem}\label{thm}
Let $\omega$, the dispersion relation, be continuous. Suppose there exists $\phi \in \mathcal{E}^\prime \setminus \{0\}$ and $t\in \mathbb{R}\setminus\{0\}$ such that $supp(U(t)\phi)$ is compact.
\begin{enumerate}
    \item For $d \geq 2$ and for each $1 \leq j \leq d$, there exists a non-open subset $Z_j \subseteq \mathbb{R}^{d-1}$ and a simply-connected open subset $\Omega_j \subseteq \mathbb{C}$ such that whenever $\xi_j^\prime \in \mathbb{R}^{d-1}\setminus Z_j$, $\xi_j \mapsto \omega(\xi_j;\xi_j^\prime)$ has an analytic extension $z_j \mapsto w(z_j;\xi_j^\prime)$ for $z_j \in \Omega_j$. More precisely for each $1 \leq j \leq d$ and $\xi_j^\prime \in \mathbb{R}^{d-1}\setminus Z_j$, there exists a countable subset $E_j\subseteq \mathbb{C}\setminus \{\eta_j=0\}$ with no accumulation points such that $\Omega_j = \mathbb{C} \setminus \bigcup\limits_{z \in E_j}\gamma_z$. If $d=1$, then there exists a countable subset $E_1\subseteq \mathbb{C}\setminus \{\eta_1=0\}$ with no accumulation points such that $\omega(\xi_1)$ has an analytic extension $\omega(z_1)$ for $z_1 \in \Omega_1 = \mathbb{C} \setminus \bigcup\limits_{z \in E_1}\gamma_z$. \label{thm1}
    \item Furthermore if $Z_j = \emptyset$ for all $1 \leq j \leq d$, which vacuously holds for $d=1$, and if $z_j\mapsto \omega(z_j;\xi_j^\prime)$ is holomorphic on some punctured disk centered at every point in $E_j$ for all $\xi_j^\prime \in \mathbb{R}^{d-1}$, then
    \begin{equation}\label{linear}
    \omega(\xi) = \sum\limits_{|\boldsymbol\alpha|\leq 1}a_{\boldsymbol\alpha }\xi^{\boldsymbol\alpha}   
    \end{equation}
    for some $a_{\boldsymbol\alpha}\in\mathbb{R}$.\label{thm2}
\end{enumerate}
\end{theorem}

Before the proof of \Cref{thm}, a discussion on the sharpness of the hypotheses of \Cref{thm} is given.

\begin{remark}\label{rmk1}
The continuity assumption of $\omega$ cannot be relaxed for $\omega$ to have an analytic extension. We give a motivating example for $d=1$.

Consider $\omega(\xi) = 1$ if $|\xi|\leq 1$ and $\omega(\xi) = 0$ otherwise. By direct computation,
\begin{equation*}
K_t(x) = \delta(x) + 2\frac{\sin x}{x}(e^{-it}-1),\ (x,t)\in \mathbb{R}^{1+1}.
\end{equation*}
Hence for every $\phi \in \mathcal{E}^\prime$ and $t \in 2\pi \mathbb{Z}$, we have
\begin{equation*}
    supp(K_t \ast \phi) = supp(\phi).
\end{equation*}

If $t \notin 2\pi\mathbb{Z}$ and $\phi \in \mathcal{E}^\prime$, we show that the compactness of $supp(K_t\ast\phi)$ implies $\phi = 0$. Note that $e^{-it\omega}\widehat{\phi}$ extends entirely and so does $\widehat{\phi}$. For all $|\xi| \leq 1$, $e^{-it\omega(\xi)}\widehat{\phi}(\xi) = e^{-it}\widehat{\phi}(\xi)$, and therefore the equality holds for all $\xi\in\mathbb{R}$ by analytic continuation. Reasoning similarly on $|\xi|>1$, we obtain
\begin{equation*}
    e^{-it\omega(\xi)}\widehat{\phi}(\xi) = e^{-it}\widehat{\phi}(\xi) = \widehat{\phi}(\xi),\ \forall \xi\in\mathbb{R},
\end{equation*}
and hence the claim. For $d \geq 2$, consider $\Tilde{\omega}:\mathbb{R}^d \rightarrow \mathbb{R}$ defined by $\Tilde{\omega}(\xi_1,\dots,\xi_d) = \omega(\xi_d)$ where $\omega$ is given as above for $d=1$.
\end{remark}

\begin{remark}
To elucidate \Cref{thm} \eqref{thm2}, let $d=1$. If a non-polynomial $\omega$ has an analytic extension to $\mathbb{C}\setminus E$, with $E$ discrete, then $supp(U(t)\phi)$ is not compact for any $\phi \in \mathcal{E}^\prime \setminus \{0\}$ and $t \neq 0$. This conclusion need not hold if the singularities of $\omega$ are not isolated.

As an example, consider $\omega(z) = -2\tan^{-1}z = i \log (\frac{i-z}{i+z})$ for $z \in \mathbb{C}$ where the branch cuts are $\gamma_i \cup \gamma_{-i} = \{(\xi,\eta):\xi = 0,\ |\eta|\geq 1\}$. Let $F(z):= e^{-it\omega(z)} = (\frac{i-z}{i+z})^t$. If $\phi \in \mathcal{E}^\prime \setminus \{0\}$ and $U(t)\phi$ has compact support, then $F(z)g(z)$ is entire by \Cref{PWS}, where $g(z)$ is the entire extension of $\widehat{\phi}$, and therefore $F(z)$ must be meromorphic. Hence for any $t \notin \mathbb{Z}$, $supp(U(t)\phi)$ is never compact for any $\phi \in \mathcal{E}^\prime \setminus \{0\}$.

For $t \in \mathbb{Z}\setminus \{0\}$, define $\phi_t$ by $\mathcal{F}[\phi_t]= g_t(z) := (z^2+1)^{|t|} h(z)$ where an entire function $h$ does not vanish at $z=\pm i$ and satisfies \Cref{exponential}. Then $supp(U(t^\prime)\phi_t)$ is compact for $t^\prime = -|t|,-|t|+1,\dots,0,\dots,|t|-1,|t|$, and otherwise, not compact. If $g(\pm i)\neq 0$, then $supp(U(t)\phi)$ compact only if $t=0$.
\end{remark}

By \Cref{PWS}, our focus is on entire functions of exponential type. Recall that entire functions that grow as polynomials are indeed polynomials.

\begin{lemma}[{\cite[Lemma 5.5]{stein2010complex}}]\label{lm}
Let $p(z) = u(z)+iv(z)$ be entire where $z\in\mathbb{C}$. Suppose there exists $C,s>0$ and a positive sequence $\rho_n \xrightarrow[n\to \infty]{}\infty$ that satisfies $v(z) \leq C \rho_n^s$ (or $v(z) \geq -C \rho_n^s$) whenever $|z|=\rho_n$ for all $n \in \mathbb{N}$. Then $p$ is a polynomial of degree at most $s$. 
\end{lemma}

\begin{lemma}\label{poly2}
Let $\xi\mapsto \omega(\xi)\in\mathbb{R}$ be continuous on $\mathbb{R}^d$. For all $\xi_j^\prime \in\mathbb{R}^{d-1},\ 1 \leq j \leq d$, assume that $\xi_j\mapsto \omega(\xi_j;\xi_j^\prime)$ has an entire extension $z_j\mapsto \omega(z_j;\xi_j^\prime)$, which is further assumed to be a linear polynomial in $z_j \in \mathbb{C}$. Then $\omega$ has an entire extension on $\mathbb{C}^d$ given by
\begin{equation}\label{homogeneous}
    \omega(z) = \sum_{\boldsymbol\alpha_j \in \{0,1\},\ 1 \leq j \leq d} a_{\boldsymbol\alpha} z^{\boldsymbol\alpha},
\end{equation}
where $a_{\boldsymbol\alpha} \in \mathbb{R}$.
\end{lemma}
\begin{proof}
The proof is by induction. Let $d=2$. For any $\xi_2 \in \mathbb{R}$, an entire extension in the first variable implies
\begin{equation*}
    \omega(z_1,\xi_2) = A_1(\xi_2)z_1 + A_0(\xi_2).
\end{equation*}
Substituting $z_1 \in \{0,1\}$, the coefficient functions $A_1,A_0$ extend to polynomials in $z_2 \in \mathbb{C}$, and therefore $\omega$ has the desired form with $a_{\boldsymbol\alpha} \in \mathbb{C}$. That $a_{\boldsymbol\alpha}$ is real follows from $\omega(\xi) \in \mathbb{R}$ for all $\xi \in \mathbb{R}^d$.

For $d \geq 2$, $\omega$ with any fixed $\xi_d\in\mathbb{R}$ satisfies the inductive hypothesis. Then
\begin{equation*}
    \omega(z_1,\dots,z_{d-1},\xi_d) = \sum_{\boldsymbol{\alpha_j} \in \{0,1\},\ 1 \leq j \leq d-1} a_{\boldsymbol\alpha(\xi_d)} z_1^{\boldsymbol\alpha_1}\cdots z_{d-1}^{\boldsymbol\alpha_{d-1}},
\end{equation*}
Substituting $z_j \in \{0,1\}$ for $1 \leq j \leq d-1$ and extending $\omega$ to a linear polynomial in the $\xi_d$-variable, \eqref{homogeneous} follows.
\end{proof}

Entire functions that grow faster than polynomials are classified by their order. For $r>0$, define
\begin{equation}\label{minmax}
M(r) = \max_{|z|=r} |f(z)|;\ m(r) = \min_{|z|=r} |f(z)|,    
\end{equation}
where $f$ is entire on $\mathbb{C}$.
\begin{definition}
Given a single-variate entire function $f$, the order of $f$ is defined as
\begin{equation*}
    \rho :=\varlimsup_{r\to\infty} \frac{\log\log M(r)}{\log r} \in [0,\infty].
\end{equation*}
Equivalently $f$ is of order $\rho<\infty$, if and only if, for any $\epsilon>0$,
\begin{equation*}
    M(r) = O_\epsilon(\exp(r^{\rho+\epsilon}))
\end{equation*}
holds as $r \rightarrow \infty$, but not for any $\epsilon < 0$.
\end{definition}

On the other hand, non-constant entire functions cannot decay too rapidly. In the following lemma, we consider some special cases of Theorem 3.2.11, Theorem 3.3.1, and Theorem 3.6.2 of \cite{boas2011entire} to prove \Cref{thm}. Note that $\rho$ is the order of an entire function and $\{\rho_n\}$ is a positive sequence.

\begin{lemma}\label{lb entire}
Let $f$ be a single-variate entire function with order $\rho \in [0,1]$. For all $\epsilon>0$, there exists a positive sequence $\rho_n\xrightarrow[n\to\infty]{}\infty$ such that
\begin{equation*}
    m(\rho_n)>M(\rho_n)^{\cos (\pi \rho) - \epsilon}.
\end{equation*}
\end{lemma}

\begin{proof}[Proof of \Cref{thm} \eqref{thm1}]
Let $d \geq 2$ and $j=1$ without loss of generality; the statement corresponding to $d=1$ is consequential. By \Cref{PWS}, $e^{-it\omega}\widehat{\phi},\ \widehat{\phi}$ have entire extensions $f,g$, respectively. Since $g(\cdot;\xi_1^\prime)$ is entire for $\xi_1^\prime \in \mathbb{R}^{d-1}$, its zeros are isolated or $g(\cdot;\xi_1^\prime)\equiv 0$. That $\phi$ is non-zero implies $Z_1:=\{\xi_1^\prime \in \mathbb{R}^{d-1}: g(\cdot;\xi_1^\prime)\equiv 0\}$ is not open in $\mathbb{R}^{d-1}$. Then $z_1 \mapsto F(z_1;\xi_1^\prime):=\frac{f(z_1;\xi_1^\prime)}{g(z_1;\xi_1^\prime)}$ has isolated singularities for $\xi_1^\prime \in \mathbb{R}^{d-1}\setminus Z_1$. Denote $\mathcal{C}_1 = \mathcal{C}_1(\xi_1^\prime)$ by the set of poles of $F(\cdot;\xi_1^\prime)$.

Suppose $p \in \mathcal{C}_1 \cap \mathbb{R}$ is a pole of $F(\cdot;\xi_1^\prime)$. Then there exists a small neighborhood $U\subseteq \mathbb{C}$ around $p$ and a non-vanishing holomorphic function $h$ on $U$ such that
\begin{equation}\label{pole}
    F(z_1;\xi_1^\prime) = (z_1-p)^{-m} h(z_1),\ \forall z_1 \in U\setminus \{p\},
\end{equation}
for some $m\in\mathbb{N}$. Since $|F(z_1;\xi_1^\prime)|=1$ for all $z_1 \in (U \cap \mathbb{R})\setminus \{p\}$ by the definition of $f,g$, we have a contradiction. Therefore $\mathcal{C}_1\cap \mathbb{R} = \emptyset$. 

We claim that $p \in \mathcal{C}_1$ if and only if $\overline{p}$ is a zero of $F(\cdot;\xi_1^\prime)$. Let $E_1 = \mathcal{C}_1 \cup \overline{\mathcal{C}_1}$ where $\overline{\mathcal{C}_1} = \{\overline{z}\in\mathbb{C}: z \in \mathcal{C}_1\}$. Since $F(\cdot;\xi_1^\prime)$ is meromorphic, so is $z_1\mapsto\overline{F(\overline{z_1};\xi_1^\prime)}$. Noting that $|F(z_1;\xi_1^\prime)|=1$ for all $z_1 \in \mathbb{R}$, it follows by analytic continuation that
\begin{equation}\label{product}    F(z_1;\xi_1^\prime)\overline{F(\overline{z_1};\xi_1^\prime)}=1,\ \forall z_1 \in \mathbb{C}\setminus E_1.
\end{equation}
In particular, $F(\cdot;\xi_1^\prime)$ is non-zero on $\mathbb{C}\setminus E_1$. If $p \in \mathcal{C}_1$, then $F(z_1;\xi_1^\prime)$ has a meromorphic representation as \eqref{pole} on some punctured disk centered at $p$, and therefore by \eqref{product},
\begin{equation*}    h(z_1)\overline{F(\overline{z_1};\xi_1^\prime)} = (z_1-p)^m. 
\end{equation*}
Taking $z_1 \rightarrow p$, we have $F(\overline{p};\xi_1^\prime)=0$ for $\overline{p} \in \overline{\mathcal{C}_1}$.

Since $F(\cdot;\xi_1^\prime)$ is non-zero on the simply-connected set $\Omega_1 = \mathbb{C} \setminus \bigcup\limits_{z \in E_1}\gamma_z$, there exists a holomorphic function $W$ on $\Omega_1$ such that $e^{-itW(z_1)}=F(z_1;\xi_1^\prime)$ for all $z_1 \in \Omega_1$. Restricting to $z_1=\xi_1 \in \Omega_1\cap\mathbb{R} = \mathbb{R}$, we have $e^{-itW(\xi_1)} = e^{-it\omega(\xi_1;\xi_1^\prime)}$, and therefore
\begin{equation*}
    \Re W(\xi_1) - \omega(\xi_1;\xi_1^\prime) = \frac{2\pi k(\xi_1)}{t},\ \Im W(\xi_1)=0, 
\end{equation*}
for some $k(\xi_1)\in\mathbb{Z}$. By continuity in $\xi_1$, $k$ is independent of $\xi_1$. Defining $\Tilde{W}(z_1) = W(z_1) - \frac{2\pi k}{t}$, we have
\begin{equation*}
\begin{aligned}
e^{-it\Tilde{W}(z_1)} &= F(z_1,\xi_1^\prime),\ \forall z_1 \in \Omega_1,\\
\Tilde{W}(\xi_1) &= \omega(\xi_1;\xi_1^\prime),\ \forall \xi_1 \in \mathbb{R}. 
\end{aligned}
\end{equation*}
Hence $\Tilde{W}$ is the analytic continuation of $\omega(\cdot;\xi_1^\prime)$. 
\end{proof}

\begin{proof}[Proof of \Cref{thm} \eqref{thm2}]
We claim $\mathcal{C}_1 = \emptyset$. By contradiction, let $p \in \mathcal{C}_1$. Then for any $\xi_1^\prime \in \mathbb{R}^{d-1}$, $F(z_1;\xi_1^\prime)$ has a meromorphic representation for $z_1 \in U \setminus \{p\}$ by \eqref{pole}. Let $\omega(z_1;\xi_1^\prime) = u(z_1)+iv(z_1)$ denote the real and imaginary parts of $\omega$. Taking the modulus and logarithm on both sides of \eqref{pole}, obtain
\begin{equation*}
\begin{aligned}
e^{tv(z_1)} &= |z_1-p|^{-m} |h(z_1)|,\\
v(z_1) &= -\frac{m}{t} \log |z_1-p| + \frac{1}{t}\log |h(z_1)|,\ \forall z_1 \in U \setminus \{p\}.
\end{aligned}
\end{equation*}

Since the range of $\omega(\cdot;\xi_1^\prime)$ on $U \setminus \{p\}$ is not dense, $\omega(\cdot;\xi_1^\prime)$ must have a pole at $p\in\mathcal{C}_1$ by the Casorati–Weierstrass Theorem. Then $\omega(\cdot;\xi_1^\prime)$ is of the form on the RHS of \eqref{pole}, which implies $F(\cdot;\xi_1^\prime)$ has an essential singularity at $p\in\mathcal{C}_1$, a contradiction. Hence $\mathcal{C}_1 = \emptyset$ and therefore $\omega(\cdot;\xi_j^\prime)$ has an entire extension for every $\xi_j^\prime \in \mathbb{R}^{d-1}$. 

Consider the first case where $\rho \in [\frac{1}{2},1]$ is the order of $g$. By \Cref{PWS}, there exists $C_i>0,\ N_i \in \mathbb{Z},\ r_i>0$ for $i=1,2$ such that
\begin{equation}\label{PWS2}
\begin{aligned}
e^{tv(z_1)}|g(z_1;\xi_1^\prime)| &\leq C_1 (1+|z_1|+\sum_{j=2}^d |\xi_j|)^{N_1} e^{r_1 |\eta_1|},\\
|g(z)| &\leq C_2 \langle z \rangle^{N_2} e^{r_2 |\Im(z)|}.
\end{aligned}    
\end{equation}
for $z \in \mathbb{C}^d$. Applying definition \eqref{minmax} to $z_1 \mapsto g(z_1;\xi_1^\prime)$, obtain
\begin{equation*}
    M(r) \leq C_2 (1+r+\sum_{j=2}^d |\xi_j|)^{N_2} e^{r_2 r},
\end{equation*}
for every $r>0$. By \Cref{lb entire}, there exists $a>0$ and a positive sequence $\rho_n \xrightarrow[n\to\infty]{}\infty$ such that
\begin{equation*}
    m(\rho_n) > M(\rho_n)^{-a} \geq C_2^{-a}(1+\rho_n+\sum_{j=2}^d |\xi_j|)^{-aN_2}e^{-ar_2 \rho_n}, 
\end{equation*}
which, combined with the first inequality of \eqref{PWS2} yields
\begin{equation}\label{PWS3}
e^{tv(z_1)} \leq C_1C_2^{a} (1+\rho_n+\sum_{j=2}^d |\xi_j|)^{N_1+aN_2} e^{(r_1+ar_2)\rho_n}, 
\end{equation}
for all $|z_1| = \rho_n,\ n \in \mathbb{N}$. By \Cref{lm}, $\omega(z_j;\xi_j^\prime)$ is a linear polynomial in $z_j$ for $1 \leq j \leq d$. Furthermore $\omega$ is of the form
\begin{equation}\label{homogeneous2}
\omega(z) = \sum_{\boldsymbol\alpha_j \in \{0,1\},\ 1 \leq j \leq d} a_{\boldsymbol\alpha} z^{\boldsymbol\alpha}   
\end{equation}
for $a_{\boldsymbol\alpha} \in \mathbb{R}$ by \Cref{poly2}.

If $\rho \in [0,\frac{1}{2})$, another application of \Cref{lb entire} yields $m(\rho_n) > M(\rho_n)^{a}$ for some $a>0$ and $\rho_n \xrightarrow[n\rightarrow \infty]{} \infty$. By the maximum modulus principle, there exists a subsequence of $\{\rho_n\}$ and $\delta > 0$ such that, without relabeling, $M(\rho_n) \geq \delta$; otherwise, $g$ is constant by the Liouville's Theorem (see \Cref{rmk_liouville}). Reasoning as \eqref{PWS3}, we have
\begin{equation*}
    e^{t v(z_1)} \leq \frac{C_1}{\delta^a} (1+\rho_n + \sum_{j=2}^d |\xi_j|)^{N_1} e^{r_1 \rho_n},
\end{equation*}
along the radii $|z_1| = \rho_n$, and hence \eqref{homogeneous2} by \Cref{poly2}.

To show $| \boldsymbol\alpha| \leq 1$ in each summand of \eqref{homogeneous2}, we argue by contradiction. Suppose there exists a multi-index $\boldsymbol\alpha$ such that $\boldsymbol\alpha_{j_1}=\boldsymbol\alpha_{j_2}=1$ for some $j_1 \neq j_2$. Then there exists $\{\lambda_j\}_{j=2}^d$ for $\lambda_j \in \mathbb{R}$ such that the function $\Tilde{\omega}(z_1) = \omega(z_1,\lambda_2 z_1,\dots,\lambda_d z_1)$ is a polynomial in $z_1$ of degree $2 \leq k \leq d$. Then
\begin{equation}\label{poly3}
\Im \Tilde{\omega}(z_1) = \sum_{n+m\leq k,\ n,m \geq 0} c_{n,m} \xi_1^n \eta_1^m,\ c_{n,m}\in\mathbb{R},   
\end{equation}
where there exists $n_0,m_0 \geq 0$ such that $n_0 + m_0 = k,\ c_{n_0,m_0}\neq 0$. 

Applying \Cref{PWS} and \Cref{lb entire} on $z_1 \mapsto \Tilde{g}(z_1):=g(z_1,\lambda_2z_1,\dots,\lambda_d z_1)$, we have
\begin{equation*}
    |\Tilde{g}(z_1)|^{-1} \leq A e^{B|\eta_1|},\: A,B>0,
\end{equation*}
for all $|z_1| = \rho_n$ for some positive sequence $\rho_n$ increasing to infinity. For all $|z_1| = \rho_n$, we argue as \eqref{PWS3} to obtain
\begin{equation*}
    e^{t \Im \Tilde{\omega}(z_1)} \leq \Tilde{A} e^{\Tilde{B}|\eta_1|},\: \Tilde{A},\Tilde{B}>0.
\end{equation*}

By a further inspection of the homogeneous polynomial of degree $k$ in \eqref{poly3}, there exists $\lambda \in \mathbb{R}$ such that $t\Im \Tilde{\omega}(\xi_1,\lambda \xi_1)$ defines a polynomial in $\xi_1$ of degree $k$ that diverges to $+\infty$ as $\xi_1 \rightarrow \infty$ (or $\xi_1 \rightarrow -\infty$). Taking this limit restricted to $\{|z_1| = \rho_n\}$ for all $n\in\mathbb{N}$, it is shown that $k \leq 1$. 
\end{proof}

\begin{remark}\label{rmk_liouville}
Regurgitating the proof of \Cref{thm} with $\phi = \delta$ readily yields the proof of \Cref{fundamental_cpt}, manifesting the dispersive nature of \eqref{FSE2} via the instant loss of compact support of the fundamental solution.    
\end{remark}

\section{Unique continuation in the multi-dimensional nonlinear regime}\label{nonlinearization}

In this section, $B>0$ is fixed. Given $z \in \mathbb{C}^d$, denote $z = \xi + i \eta = \Re z + i \Im z$ and $|z| := \max\limits_{1 \leq j \leq d} |z_j|$. Define
\begin{equation}\label{setE}
\begin{split}
    u^*(\xi) &= \sup_{t \in I} |\widehat{u(t)}(\xi)|;\ a(\xi) = \sup_{|\xi^\prime| \geq |\xi|} u^*(\xi^\prime),\\
    E_{Q,c} &= \{\xi \in \mathbb{R}^d: a(\xi) > c \underbrace{a \ast \cdots \ast a}_{k} (\xi),\ a(\xi) > e^{-\frac{|\xi|}{Q}},\ a(\xi) = u^*(\xi)\},  
\end{split}
\end{equation}
where $u \in C([-T,T];H^s(\mathbb{R}^d))$ satisfies $\sup\limits_{t \in I} \| u(t) \|_{H^s} \leq M < \infty$ with $0 < T < \infty$ and
\begin{equation}\label{duhamel}
        u(t_2) = U(t_2 - t_1)u(t_1) - i \int_{t_1}^{t_2} U(t_2 - \tau) N[u](\tau) d\tau
\end{equation}
for some $s>d$, for any $t_1, t_2 \in [-T,T]$, and $N[u] = u^k$ (or any k-fold product for $k \geq 2$ involving complex conjugates). By definition, $a$ is decreasing in $|\xi|$ and
\begin{equation*}
u^*(\xi),\ a(\xi) \lesssim M B^d \langle \xi \rangle^{-s} \text{ for all } \xi \in \mathbb{R}^d.  
\end{equation*}
Observe that for any $\xi \in \mathbb{R}^d$, there exists $t_1 = t_1(\xi) \in [-T,T]$ such that $u^*(\xi) = |\widehat{u(t_1)}(\xi)|$. Then let $t_2 \in [-T,T]$ such that $|\Delta t| := |t_2 - t_1| \geq \epsilon_0 T$ for $\epsilon_0 \ll 1$. Denote $\mu = sgn(\Delta t) \in \{\pm 1\}$.

\begin{theorem}\label{thm_nonlinear}
Consider an entire function $\omega$ on $\mathbb{C}^d$ that is real-valued when restricted to $\mathbb{R}^d$. Let $u \in C([-T,T];H^s(\mathbb{R}^d))$ satisfy \eqref{duhamel} and $\sup\limits_{t \in I} \| u(t) \|_{H^s} \leq M$, and assume $supp (u(t)) \subseteq \{x \in \mathbb{R}^d: |x| < B\}$ for all $t \in [-T,T]$. Let $Q \gg B$. Suppose there exists a sequence $\{\xi_n\} \subseteq E_{Q,c}$ where for each $\xi_n$, there exists $\eta_n \in \mathbb{R}^d$ with $|\eta_n| \simeq |\xi_n|^{-1}$ such that
\begin{equation}\label{growth}
    (t_2(\xi_n) - t_1(\xi_n)) \Im \omega(\xi_n + i \eta_n) >0,\ |\Im \omega(\xi_n + i \eta_n)| \gtrsim \frac{|\xi_n|}{T}  
\end{equation}
for all sufficiently large $n$ as $|\xi_n| \xrightarrow[n \rightarrow \infty]{} \infty$. Then $u \equiv 0$ on $[-T,T] \times \mathbb{R}^d$.   
\end{theorem}

Indeed $E_{Q,c}$ contains an unbounded sequence.

\begin{lemma}\label{l:3.1}
    Let $Q > 0$. Then there exists $c > 0$, independent of $Q$, such that $E_{Q,c}$ is unbounded in $\mathbb{R}^d$.
\end{lemma}

Proofs of \Cref{l:3.1} in \cite{bourgain1997compactness,carvajal2005unique,panthee2005unique} assumed that $u(x,0)$ is sufficiently smooth such that $|a(\xi)| \lesssim \langle \xi \rangle^{-4}$ in $d = 1,2$. For a general dimension, it suffices to have $|a(\xi)| \lesssim \langle \xi \rangle^{-s}$ for any $s > d$. Then $a \in L^1(\mathbb{R}^d)$ from which the argument follows similarly as (1.10) of \cite{bourgain1997compactness}. Likewise the following lemma can be shown similarly as \cite{bourgain1997compactness,carvajal2005unique,panthee2005unique}. More precisely, \cite{panthee2005unique} extended the derivative estimate of single-variate entire functions to entire functions of two complex variables by iterating the proof in \cite{bourgain1997compactness} for each variable while fixing the rest. This technique generalizes to any dimension.

\begin{lemma}\label{l:3.2}
    Let $\Phi:\mathbb{C}^d \rightarrow \mathbb{C}$ be entire such that $|\Phi(z)| \lesssim e^{C B |\Im z|}$ for all $z \in \mathbb{C}^d$ for some $C>0$ and assume $z_j \mapsto \Phi(z_j;z^\prime_j)$ is bounded and absolutely integrable on the real line for all $1 \leq j \leq d$ (see \eqref{inclusion} for notation). If    \begin{equation}\label{l:3.2_hypothesis}
        |\eta| \leq B^{-1} \left( 1 + |\log \sup_{|\xi^\prime| \geq |\xi|} |\Phi(\xi^\prime)||\right)^{-1},    
    \end{equation}
    then
    \begin{equation*}
        \sup_{|\xi^\prime| \geq |\xi|} |\nabla \Phi (\xi^\prime + i \eta)| \lesssim_d B \sup_{|\xi^\prime| \geq |\xi|} | \Phi (\xi^\prime)| \left( 1 + |\log \sup_{|\xi^\prime| \geq |\xi|} |\Phi(\xi^\prime)||\right).
    \end{equation*}
\end{lemma}

\begin{proof}[Proof of \Cref{thm_nonlinear}]

Let $\xi \in E_{Q,c}$ and choose $\eta$ such that \eqref{growth} is satisfied. Then we have $t_1 \in [-T,T]$ such that $|\widehat{u(t_1)}(\xi)| = u^*(\xi) = a(\xi)$. Take the Fourier transform of \eqref{duhamel} and change variables $\tau \mapsto \tau + t_1$ to obtain
\begin{equation}\label{duhamel2}
    e^{i \Delta t \omega(z)} \widehat{u(t_2)}(z) = \widehat{u(t_1)}(z) - i \int_0^{\Delta t} e^{i \tau \omega(z)} \widehat{N[u](t_1 + \tau)}(z) d\tau,
\end{equation}
where the entire extension $\xi \mapsto z = \xi + i \eta$ is by \Cref{PWS}. The LHS of \eqref{duhamel2} is estimated as
\begin{equation}\label{duhamel_upperbound}
    \left|e^{i \Delta t \omega(z)} \widehat{u(t_2)}(z)\right| \leq e^{-\Delta t \Im \omega(z)} \int_{|x| \leq B} |u(t_2,x)| e^{B |\eta|} dx \lesssim_{d} MB^d e^{B |\eta| - \Delta t \Im \omega(z)},
\end{equation}
by the compactness of $supp(u(t_2))$. The RHS is estimated from below by
\begin{equation*}
\begin{split}
    \left|\widehat{u(t_1)}(z) - i \int_0^{\Delta t} e^{i \tau \omega(z)} \widehat{N[u](t_1 + \tau)}(z)  
    d\tau \right| &\geq |\widehat{u(t_1)}(\xi)| - \int_0^{|\Delta t|} e^{-\mu \tau \Im \omega(z)} \left| \widehat{N[u](t_1 + \mu \tau)} (\xi) \right| d\tau\\
    &- |\widehat{u(t_1)}(z) - \widehat{u(t_1)}(\xi)|\\
    &- \int_0^{|\Delta t|} e^{-\mu\tau \Im \omega(z)} \left| \widehat{N[u](t_1 + \mu \tau)} (z)-\widehat{N[u](t_1 + \mu \tau)} (\xi) \right| d\tau\\
    &=: I - II - III.    
\end{split}
\end{equation*}
Since
\begin{equation*}
    \int_0^{|\Delta t|} e^{-\mu\tau \Im \omega(z)} \left| \widehat{N[u](t_1 + \mu \tau)} (\xi) \right| d\tau \leq \int_0^{|\Delta t|} e^{-\mu\tau \Im \omega(z)} a \ast \cdots \ast a (\xi) d\tau < \frac{1-e^{-\mu |\Delta t| \Im \omega(z)}}{c \mu \Im \omega(z)} a(\xi),
\end{equation*}
we have
\begin{equation}\label{estimate:I}
    I > \left( 1 - \frac{1-e^{-\mu|\Delta t| \Im \omega(z)}}{c \mu \Im \omega(z)}\right) a(\xi).
\end{equation}
Let $\Phi(z) = \widehat{u(t_1)}(z)$. Then,
\begin{equation}\label{estimate:II}
    II = |\Phi(\xi + i \eta) - \Phi(\xi)| \leq |\eta| \sup_{|\eta^\prime| \leq |\eta|} |\nabla \Phi (\xi + i \eta^\prime)| \lesssim a(\xi),
\end{equation}
where the last inequality follows from \Cref{l:3.2}. Since
\begin{equation*}
    1 + |\log \sup_{|\xi^\prime| \geq |\xi|} |\Phi(\xi^\prime)|| \leq 1 + |\log a(\xi)| \leq 1 + \frac{|\xi|}{Q},
\end{equation*}
the estimate \eqref{l:3.2_hypothesis} is justified for $Q > 2B$. The estimation of the last term follows similarly as
\begin{align}\label{estimate:III}
    III &\leq \int_0^{|\Delta t|} e^{-\mu \tau \Im \omega(z)} |\eta| \biggl\{\int \sup_{|\eta^\prime| \leq |\eta|} |\nabla \Phi (\xi - \xi_1 -\cdots - \xi_{k-1} + i \eta^\prime)| |\Phi(\xi_1) \cdots \Phi(\xi_{k-1})| d\xi_1 \cdots d\xi_{k-1} \biggl\} d\tau\nonumber\\
    &\lesssim_k \int_0^{|\Delta t|} e^{-\mu \tau \Im \omega(z)} a(\xi) d\tau = \frac{1-e^{-\mu |\Delta t| \Im \omega(z)}}{c\mu \Im \omega(z)} a(\xi),
\end{align}
where the details leading up to \eqref{estimate:III} could be found in \cite{bourgain1997compactness,carvajal2005unique,panthee2005unique}. Combining \eqref{duhamel_upperbound}, \eqref{estimate:I}, \eqref{estimate:II}, and \eqref{estimate:III}, we have
\begin{equation}\label{exp_decay}
    2 e^{-C_0 |\xi|} > 2e^{- \Delta t \Im \omega(z)} > \frac{a(\xi)}{4} > \frac{e^{-\frac{|\xi|}{Q}}}{4},
\end{equation}
for some $C_0 > 0$ independent of $Q$. Take $Q \gg \max (B,C_0^{-1})$ and $|\xi|$ arbitrarily large by \eqref{growth} to derive a contradiction unless $u \equiv 0$.
\end{proof}

\begin{remark}\label{rmk:3.1}
    The growth condition of $\omega$ \eqref{growth} yields the exponential decay of the LHS of \eqref{exp_decay}, which gives the desired trivial solution in the proof. If $\omega$ is a polynomial of degree $n$, then the technique of this section applies if and only if $n \geq 3$, as can be justified by the linear expansion, 
    \begin{equation*}
        |\Im \omega(\xi + i \eta)| \simeq |\nabla \omega(\xi)\eta| \simeq |\xi|^{n-1}|\eta| \simeq |\xi|^{n-2},
    \end{equation*}
    for $|\xi| |\eta| \simeq 1$. Hence \eqref{growth} holds if and only if $n \geq 3$. Bourgain's argument in \cite[Equation (3.11), p.447]{bourgain1997compactness} depended heavily on $n=3$ (generalized KdV).  
\end{remark}

\section{Full domain as the support}\label{full_domain}

Here we restrict attention to the dynamics on the real line. Our goal is to show
\begin{equation}\label{full}
    supp(U(t)\phi) = \mathbb{R}\ \text{for all}\ \phi \in \mathcal{E}^\prime(\mathbb{R}) \setminus \{0\}.
\end{equation}

Note that this is a stronger statement than the non-compactness of $supp(U(t)\phi)$. It is of interest to fully classify the dispersion relations that satisfy \eqref{full}. It is of further interest to study an analogous question for a more general spatial domain.

It suffices to show the real-analyticity of $K_t = \mathcal{F}^{-1}[e^{-it\omega}]$ though such strong property is not necessary as can be seen in the first example of \eqref{half-wave2} where $\omega(\xi) = |\xi|$. Real-analytic functions can be locally extended to holomorphic functions defined in a region in $\mathbb{C}$, and therefore if the real zeros of a real-analytic function accumulate, then the function must be identically zero. Hence if $K_t$ is real-analytic, then $supp(K_t)=\mathbb{R}$. Moreover if $\phi \in \mathcal{E}^\prime$, then $K_t \ast \phi \in \mathcal{S}^\prime$ is real-analytic, and hence \eqref{full}. A useful tool in understanding the microlocal properties of a distribution is the following wave packet transform; see \cite{folland2016harmonic}.

\begin{definition}
Let $\tau \geq 0,\ x,\xi\in\mathbb{R}$. The FBI transform of $f$ is
\begin{equation*}
    P^\tau f(x,\xi) = \int_\mathbb{R} e^{-iy\xi}e^{-\tau(y-x)^2}f(y)dy.
\end{equation*}
\end{definition}

By direct computation,
\begin{equation*}
    P^\tau f(x,\xi) = \sqrt{\frac{\pi}{\tau}}\int \widehat{f}(\xi-z)e^{-ixz}e^{-\frac{z^2}{4\tau}}dz.
\end{equation*}
If $\widehat{f}(\xi) = \mathcal{F}[K_t](\xi) = e^{-itw(\xi)}$, then
\begin{equation}\label{RA}
\begin{aligned}
P^\tau f(x,\tau\xi) &= \sqrt{\pi} \int e^{-i(tw(\tau \xi - \sqrt{\tau}z) + \sqrt{\tau}xz)}e^{-\frac{z^2}{4}}dz\\
&=:\sqrt{\pi}I(\sqrt{\tau}).
\end{aligned}
\end{equation}

\begin{lemma}[{\cite[Theorem 3.31]{folland2016harmonic}}]
Consider $I(\lambda)$ in \eqref{RA} where $\lambda = \sqrt{\tau}$. Given $x_0 \in \mathbb{R},\ t \in \mathbb{R}\setminus\{0\}$, suppose there exist a neighborhood of $x_0$, say $U \subseteq \mathbb{R}$, and constants $C,\sigma,M,\lambda_0>0$ such that
\begin{equation}\label{FBI}
    |I(\lambda)| \leq C e^{-\sigma \lambda^2},
\end{equation}
for all $|\xi|>M,\ x \in U,\ \lambda>\lambda_0$. Then $K_t \in \mathcal{S}^\prime(\mathbb{R})$ is real-analytic at $x_0$ and conversely.
\end{lemma}

For all $\omega(\xi)$ that satisfies \eqref{FBI} for any $x_0 \in \mathbb{R}$, \eqref{full} holds. For now, we content ourselves with showing \eqref{FBI} for polynomial dispersion relations by applying the method of steepest descent.
\begin{proposition}
For every polynomial $\omega$ of degree $\geq 2$, $I(\lambda)$ satisfies \eqref{FBI} for every $x_0 \in\mathbb{R},\ t \in \mathbb{R}\setminus\{0\}$.
\end{proposition}
\begin{proof}
Suppose
\begin{equation*}
    \omega(z) = a_N z^N + \sum\limits_{n=0}^{N-1}a_n z^n=:a_N z^N + \omega_1(z),
\end{equation*}
where $N \geq 2,\ a_N \neq 0,\ a_n\in\mathbb{R}$ and $z\in\mathbb{C}$. Assume $a_N=1$ and $t>0$ without loss of generality. Observe that
\begin{equation*}
I(\lambda) = \int_{-\infty}^\infty e^{\lambda \Phi(z)}e^{-\frac{z^2}{4}}dz,\ \Phi(z) = -\frac{i}{\lambda}\left(tw(\lambda^2\xi - \lambda z)+\lambda x z\right),
\end{equation*}
where $|x-x_0|< \delta$ for some $\delta>0$ to be determined. A saddle point $z_0 \in \mathbb{C}$ of $\Phi$ is a root of $\Phi^\prime(z)=0$, which is equivalent to
\begin{equation*}
\omega^\prime(\Tilde{z}) - \frac{x}{t} = f(\Tilde{z})+g(\Tilde{z})=0,    
\end{equation*}
where $\prime=\frac{d}{dz},\ \Tilde{z} = \lambda^2\xi - \lambda z$ and
\begin{equation*}
f(\Tilde{z}) = N \Tilde{z}^{N-1} - \frac{x}{t},\ g(\Tilde{z}) = \omega_1^\prime(\Tilde{z}).    
\end{equation*}

The roots of $f$ are given by
\begin{equation}\label{roots of u}
\begin{aligned}
\Tilde{z} &= \left(\frac{|x|}{Nt}\right)^{\frac{1}{N-1}}e^{i\left(\phi_x + \frac{2\pi j}{N-1}\right)},\ j=0,1,\dots N-2,\\
\phi_x &= \begin{cases}
0, & x\geq 0\\
\pi, & x<0.
\end{cases}
\end{aligned}
\end{equation}
Consider $\{z\in\mathbb{C}: |z-\lambda \xi| = \frac{c}{\lambda}\} = \{\Tilde{z} \in \mathbb{C}: |\Tilde{z}| = c\}$ where $c = c(x_0,t,w)> \left(\frac{|x_0|}{Nt}\right)^{\frac{1}{N-1}}$ satisfies
\begin{equation}\label{Rouche}
    \sum_{n=1}^{N-1} n |a_n| c^{n-1} \leq \frac{N}{2}c^{N-1} - \frac{|x|}{2t},
\end{equation}
for all $|x-x_0| < \delta$ for some $\delta = \delta(c)>0$. By shrinking $\delta$, if necessary, further assume
\begin{equation}\label{roots of u2}
c>\left(\frac{|x|}{Nt}\right)^{\frac{1}{N-1}}    
\end{equation}
for all $|x-x_0|<\delta$. Then
\begin{equation*}
|g(\Tilde{z})|\leq \frac{|f(\Tilde{z})|}{2}<|f(\Tilde{z})|,    
\end{equation*}
for all $|\Tilde{z}|=c$ by \eqref{Rouche}. Since $f$ has all complex roots in the disk $|z-\lambda \xi| < \frac{c}{\lambda}$ by \eqref{roots of u}, \eqref{roots of u2}, it follows that $f+g$ has $N-1$ roots, counting multiplicities, in the same region by the Rouch\'e's Theorem. Hence a saddle point has the form
\begin{equation*}
    z_0 = \lambda\xi + \frac{r}{\lambda}e^{i\phi},
\end{equation*}
for some $0\leq r<c$ and $\phi \in [0,2\pi)$.

By our hypothesis on $\omega$, there exists $m \geq 2$ such that
\begin{equation*}
    \Phi(z) = \Phi(z_0) + \sum_{n=m}^N \frac{\Phi^{(n)}(z_0)}{n!}(z-z_0)^n.
\end{equation*}
To obtain the dominant term of the asymptotics of $I(\lambda)$ as $\lambda \rightarrow \infty$, consider an approximate constant-phase contour
\begin{equation*}
    \gamma=\{\Im\left(\frac{\Phi^{(m)}(z_0)}{m!}(z-z_0)^m\right)=0\},
\end{equation*}
that passes through some saddle point; if this contour contains multiple saddle points, one can apply the method of steepest descent finitely many times by taking an appropriate partition of unity, and therefore we may assume that $\gamma$ contains a unique saddle point $z_0$.

Suppose $m$ is even. With a change of variable $z=z_0 + \rho e^{i\theta}$, $z \in \gamma$ if and only if
\begin{equation}\label{contour3}
    \sin\left(m\theta + arg( \Phi^{(m)}(z_0))\right)=0,
\end{equation}
where $arg(\Phi^{(m)}(z_0)) \in [0,2\pi)$. To define a convergent integral, let
\begin{equation*}
\theta_0 = \frac{\pi-arg(\Phi^{(m)}(z_0))}{m},\ \theta_1 = \theta_0 + \pi
\end{equation*}
so that 
\begin{equation}\label{contour4}
    \Re\left(\frac{\Phi^{(m)}(z_0)}{m!}(z-z_0)^m\right) = -\frac{|\Phi^{(m)}(z_0)|}{m!}\rho^m,
\end{equation}
for all $z=z_0 + \rho e^{i\theta_j},\ j=0,1$. Hence the desired contour with a convergent integrand is
\begin{equation}\label{cauchy}
\gamma = \{z =z_0+ \rho e^{i\theta_0}: 0<\rho<\infty\}\cup \{z = z_0 + \rho e^{i\theta_1}: \infty<\rho<0\}.    
\end{equation}
Since the integrand of $I(\lambda)$ is entire, $I(\lambda)$ may be evaluated on $\gamma$ instead of $(-\infty,\infty)$ by the Cauchy's Theorem. Then,
\begin{equation}\label{steep}
\begin{aligned}
I(\lambda) &\sim e^{\lambda \Phi(z_0)-\frac{z_0^2}{4}+i\theta_0} \int_{-\infty}^\infty e^{-\frac{\lambda|\Phi^{(m)}(z_0)|}{m!}\rho^m}d\rho,\ \lambda\rightarrow\infty,\\
&= e^{\lambda \Phi(z_0)-\frac{z_0^2}{4}+i\theta_0}\Big(\frac{|\lambda\Phi^{(m)}(z_0)|}{m!}\Big)^{-\frac{1}{m}}\int_{-\infty}^\infty e^{-\rho^m}d\rho.
\end{aligned}
\end{equation}

Alternatively suppose $m\geq 3$ is odd. To find an approximate constant-phase contour, $\theta$ satisfies \eqref{contour3}, \eqref{contour4}. Define
\begin{equation*}
    \theta_0 = \frac{\pi-arg(\Phi^{(m)}(z_0))}{m},\ \theta_1 = \frac{3\pi-arg(\Phi^{(m)}(z_0))}{m}.
\end{equation*}
Note from \eqref{cauchy} that $\gamma$ is not a straight line when $m$ is odd, and in fact $\theta_1 = \theta_0 + \frac{2\pi}{m}$. By the Cauchy's Theorem,
\begin{equation}\label{steep2}
\begin{aligned}
I(\lambda) &= \int_{\gamma_1}e^{\lambda\Phi(z)}e^{-\frac{z^2}{4}}dz+\int_{\gamma_2}e^{\lambda\Phi(z)}e^{-\frac{z^2}{4}}dz\\
&\sim e^{\lambda\Phi(z_0)-\frac{z_0^2}{4}+i\theta_0}(1-e^{i\frac{2\pi}{m}})\int_0^\infty e^{-\frac{\lambda |\Phi^{(m)}(z_0)|}{m!}\rho^m}d\rho\\
&\sim e^{\lambda\Phi(z_0)-\frac{z_0^2}{4}+i\theta_0}(1-e^{i\frac{2\pi}{m}})\left(\frac{\lambda |\Phi^{(m)}(z_0)|}{m!}\right)^{-\frac{1}{m}}\int_0^\infty e^{-\rho^m}d\rho. 
\end{aligned}
\end{equation}
It suffices to estimate the final terms of \eqref{steep}, \eqref{steep2}. We claim that
\begin{equation*}
    |e^{-\frac{z_0^2}{4}}|\leq e^{-\sigma \lambda^2},
\end{equation*}
for some $\sigma>0$. Other terms are shown to be $O(1)$ as $\lambda\rightarrow\infty$. Since
\begin{equation*}
    |\omega(\lambda^2\xi-\lambda z_0)| = |\omega(-re^{i\phi})|\leq \sum_{n=0}^N |a_n| r^n \leq \sum_{n=0}^N |a_n| c^n<\infty,
\end{equation*}
and
\begin{equation*}
\Re(i\lambda x z_0) = \Re(i\lambda^2 x \xi + ixre^{i\phi}) = \Re(ixre^{i\phi}) \leq (|x_0|+\delta)c,
\end{equation*}
there exists $M>0$ independent of $\lambda$ such that for all $\lambda>0$,
\begin{equation*}
    |e^{\lambda\Phi(z_0)}| \leq M.
\end{equation*}
Since $m \geq 2$, we have $\lambda \Phi^{(m)}(z_0) = -i(-\lambda)^m \omega^{(m)}(\lambda^2\xi-\lambda z_0)$, and therefore
\begin{equation*}
    |\lambda \Phi^{(m)}(z_0)|^{-\frac{1}{m}} = \lambda^{-1} |w^{(m)}(-re^{i\phi})|^{-\frac{1}{m}} \leq M,
\end{equation*}
where $M$ is independent of $\lambda \geq 1$. Note that $r,\phi$ depend only on $x,t,w$. Since $\Phi^{(m)}(z_0) \neq 0$, $|\omega^{(m)}(-re^{i\phi})|$ is uniformly bounded below by a positive constant for all $|x-x_0|<\delta$.

By direct computation,
\begin{equation*}
    \Re(z_0^2) = \lambda^2\xi^2 +2r \cos(\phi) \xi +\frac{r^2\cos (2\phi)}{\lambda^2}
\end{equation*}
is a quadratic polynomial in $\xi$ that obtains the global minimum $-\frac{r^2\sin^2\phi}{\lambda^2}$ at $\xi_m = -\frac{r\cos\phi}{\lambda^2}$. Hence there exist $\sigma,\lambda_0>0$ such that
\begin{equation*}
    \Re(z_0^2) \geq \sigma \lambda^2,
\end{equation*}
for all $\lambda > \lambda_0$ and $|\xi|\geq 1$.
\end{proof}

\section{Time-fractional propagation and dispersive estimates}\label{frac_trav}

%Wave phenomena, from the perspectives of the compactness of solution support, can be complicated in higher dimensions. For example, the PDE

%\begin{equation*}
    %i\partial_t u = (-i\partial_{x} - \partial_{yy})u,\ (x,y,t) \in \mathbb{R}^{2+1},
%\end{equation*}
%evolves an initial datum $\phi_1(x)\phi_2(y)$ to $\phi_1(x-t)\cdot K_t^S \ast \phi_2(y)$ where $K_t^S$ is the fundamental solution for the one-dimensional Schr\"odinger evolution and $\phi_j,\ j=1,2$, is any reasonably smooth function with a rapid decay. The solution support remains compact for all $t \in \mathbb{R}$ only in the $x$-coordinate. A deeper insight is gained by avoiding this topological complexity, and therefore, assume $d=1$.

This section is concerned with the dispersive properties of \eqref{stfsch} where it is assumed that $0<\alpha<1,\ \beta>0,\ \alpha \leq \gamma \leq 1$ unless specified. The fundamental solution is given by $K_t = \mathcal{F}^{-1}[E_\alpha (i^{-\gamma}t^\alpha|\xi|^\beta)]$ for $t \geq 0$. To ensure that $K_t$ is tempered, the bound on $\gamma$ is assumed as above, or more precisely, see \Cref{cor41}. By \Cref{fundamental_cpt}, the compactness of solution support is preserved in time for the first-order linear systems only for the transport equation when space and time scale linearly. For \eqref{stfsch}, ISP holds even for $\alpha = \beta < 1$ (when $\Delta x \sim \Delta t$) since $|\xi|$ is not analytic at the origin. When $\alpha = \beta = \gamma < 1$, however, it is shown that the solution operator is well-approximated by the half-wave operator as $t \rightarrow \infty$, and therefore, the mass of the solution, measured in the $L^2$ norm, is concentrated inside the light cone $|x| \lesssim |t|$ although the solution support loses compactness instantaneously.

Let us consider some special cases of \eqref{stfsch}. For $\alpha=\beta=\gamma=2$, \eqref{stfsch} reduces to the well-known linear wave equation whose exact solution is given by the d'Alembert's formula. For $\alpha = 1,\ \beta=2,\ \gamma = 2$, \eqref{stfsch} is the diffusion equation. For the space-time fractional Schr\"odinger equation, i.e., $\alpha = \gamma$, the explicit formula of the fundamental solution $K_{\alpha,\beta}^{(0)}(x,t)$ for \eqref{stfsch} is given in \cite[Chapter 13]{laskin2018fractional} by the Fox H-function as
\begin{equation}\label{H-function}    K_{\alpha,\beta}^{(0)}(x,t) = \frac{1}{|x|} H_{3,3}^{2,1}\left(-\frac{i^\alpha |x|^\beta}{t^\alpha}\left|
\begin{array}{c}
(1,1), (1,\alpha), (1,\frac{\beta}{2})\\
(1,\beta), (1,1), (1,\frac{\beta}{2})\\
\end{array}
\right.\right).   
\end{equation}
The $H$-function is a generalized Meijer $G$-function defined via the Mellin-Barnes integral; see Appendix A of \cite{laskin2018fractional} for an introduction to the topic. Instead of directly analyzing these special functions, our approach is to study the time evolution on the Fourier space using the Mittag-Leffler function. Let
\begin{equation*}
\widehat{K_t^{hw}}(\xi) = e^{-it|\xi|},\ U^{hw}(t)\phi:= K_t^{hw} \ast \phi.     
\end{equation*}
\begin{proposition}\label{ftw3}
For $\alpha = \beta = \gamma \in (0,1)$, the fundamental solution for \eqref{stfsch} is of the form
\begin{equation}\label{ftw4}
K_t = \frac{1}{\alpha}K_t^{hw} + R_{t,\alpha},\ t \geq 0,
\end{equation}
where $R_{t,\alpha} \in \mathcal{S}^{\prime}(\mathbb{R}^d)$, and for any $\phi \in \mathcal{E}^\prime(\mathbb{R}^d),\ \psi \in \mathcal{S}(\mathbb{R}^d)$, there exists $C(\alpha) = C(\phi,\psi,\alpha)>0$ such that for all $t>0$,
\begin{equation}\label{disp_est5}
    |\langle R_{t,\alpha} \ast \phi,\psi\rangle| \leq \frac{C(\alpha)}{t^{\alpha}},
\end{equation}
and
\begin{equation}\label{disp_est6}
    \| K_t \ast \phi \|_{L^\infty} \lesssim_\alpha |t|^{-\min(\frac{d-1}{2},\alpha)}\left(\| \phi \|_{\dot{B}^{\frac{d+1}{2}}_{1,1}} + \| \phi \|_{B^{d-\beta}_{1,1}}\right).
\end{equation}

Furthermore $E_{\alpha}(t) \xrightarrow[t \to \infty]{} 0$ strongly in $L^2(\mathbb{R}^d)$ but not uniformly where $E_\alpha(t)\phi := R_{t,\alpha} \ast \phi$.

%\begin{equation}\label{half-wave2}
%    K_t^{hw}(x) =
%    \begin{cases}
%        \frac{\delta(x-t)+\delta(x+t)}{2}+\frac{i}{2\pi}(\frac{1}{x-t}-\frac{1}{x+t}), & \ x \in \mathbb{R},\\
%        \frac{\Gamma(\frac{d+1}{2})}{\pi^{\frac{d+1}{2}}} \frac{(i|t|)^{-d}}{(1-\frac{|x|^2}{t^2})^{\frac{d+1}{2}}}, & x \in \mathbb{R}^d,\ d\geq 2,
%    \end{cases}
%\end{equation}

\end{proposition}

\begin{remark}\label{half-wave}
For $\alpha = 1,\ \beta = 1,\ \gamma=1$, \eqref{stfsch} is the half-wave equation whose solutions decay in time as $t^{-\frac{d-1}{2}}$. More precisely, the dispersive estimate
\begin{equation}\label{disp_est_hw}
    \| K_t^{hw} \ast \phi \|_{L^\infty} \lesssim |t|^{-\frac{d-1}{2}} \| \phi \|_{\dot{B}^{\frac{d+1}{2}}_{1,1}},
\end{equation}
is satisfied in the Besov space. Hence for $d \geq 3$ and $\alpha = \beta = \gamma \in (0,1)$, the solution \eqref{ftw4} decays as $t^{-\alpha}$ by \eqref{disp_est5}, independent of the spatial dimension. This could be understood as a consequence of the long-memory effect due to the Caputo derivative. For $d=2$, the dominant time decay is $t^{-\frac{1}{2}}$ if $\alpha \in (\frac{1}{2},1]$, and $t^{-\alpha}$ if $\alpha \in (0,\frac{1}{2}]$. For $d=1$, no time decay is expected since an explicit computation via the inverse Fourier transform yields
\begin{equation}\label{half-wave2}
    K_t^{hw} = \frac{\delta(x-t)+\delta(x+t)}{2}+\frac{i}{2\pi}(\frac{1}{x-t}-\frac{1}{x+t}),
\end{equation}
where the convolution against $\frac{1}{x \pm t}$ is the time-shifted Hilbert transform.

%For $d=1$, a direct computation yields terms $\frac{1}{x \pm t}$ whose convolution against $\phi \in \mathcal{E}^\prime(\mathbb{R})$ is given by the Hilbert transform. For $d \geq 2$, the (inverse) Fourier transform of $e^{-(\epsilon + it)|\xi|}$ is related to the Hankel transform of order $\frac{d}{2}-1$ whose exact formula is in \cite[Chapter 17]{poularikas2018handbook}. Taking $\epsilon \to 0$ amounts to approximating the identity operator with the Poisson kernel, and hence \eqref{half-wave2}.    
\end{remark}

The time evolution is unique if $\phi$ does not grow too rapidly and $K_t \in \mathcal{S}^\prime(\mathbb{R}^d)$. That $K_t$ is tempered is determined by the asymptotic relations of the Mittag-Leffler function. 

\begin{lemma}[{\cite[Theorem 1.3, 1.4]{podlubny1999fractional}}]\label{mlf_decay}
Let $\beta >0,\ 0<\alpha\leq \gamma \leq 1$.

If $\alpha \leq \gamma < 2\alpha$, then for every $k \in \mathbb{N}$, we have
\begin{equation*}
E_\alpha(i^{-\gamma} |\xi|^\beta) = \frac{1}{\alpha} \exp (i^{-\frac{\gamma}{\alpha}} |\xi|^{\frac{\beta}{\alpha}}) - \sum_{j=1}^k \frac{i^{\gamma j}}{\Gamma(1-\alpha j)}|\xi|^{-\beta j} + O_k (|\xi|^{-\beta (1+k)}),\ |\xi| \rightarrow \infty.
\end{equation*}
If $2 \alpha \leq \gamma \leq 1$, then
\begin{equation*}
   E_\alpha(i^{-\gamma}|\xi|^\beta) = - \sum_{j=1}^k \frac{i^{\gamma j}}{\Gamma(1-\alpha j)}|\xi|^{-\beta j} + O_k (|\xi|^{-\beta (1+k)}),\ |\xi| \rightarrow \infty.
\end{equation*}
\end{lemma}

\begin{corollary}\label{cor41}
If $\alpha \leq \gamma \leq 1$, then $K_t = \mathcal{F}^{-1}[E_\alpha (i^{-\gamma}t^\alpha|\xi|^\beta)] \in \mathcal{S}^\prime(\mathbb{R}^d)$ and
\begin{equation}\label{mlf_decay4}
\begin{split}
E_\alpha(i^{-\gamma} |\xi|^\beta) &= -\frac{i^\gamma}{\Gamma(1-\alpha)} |\xi|^{-\beta} + O(|\xi|^{-2\beta}),\ |\xi| \rightarrow \infty\ \text{if}\ \alpha < \gamma \leq 1,\\
E_\alpha(i^{-\gamma} |\xi|^\beta) &= \frac{1}{\alpha} \exp(-i |\xi|^{\frac{\beta}{\alpha}}) + O(|\xi|^{-\beta}),\ |\xi| \rightarrow \infty\ \text{if}\ \alpha = \gamma.
\end{split}    
\end{equation}
If $0 \leq \gamma < \alpha$, then $K_t \notin \mathcal{S}^\prime(\mathbb{R}^d)$.
\end{corollary}
\begin{proof}
    \eqref{mlf_decay4} follows immediately from \Cref{mlf_decay}. That $K_t$ is tempered for $\alpha \leq \gamma \leq 1$ follows from \eqref{mlf_decay4} and    
\begin{equation*}
    |E_\alpha (i^{-\gamma}t^\alpha|\xi|^\beta)| \lesssim 1,\ \xi = O(1).
\end{equation*}
For $0 \leq \gamma < \alpha$, \Cref{mlf_decay} again implies
\begin{equation*}
    E_{\alpha}(i^{-\gamma} |\xi|^\beta) = \frac{1}{\alpha} \exp \left(|\xi|^{\frac{\beta}{\alpha}} e^{-i\frac{\pi \gamma}{2 \alpha}} \right) + O(|\xi|^{-\beta}),\ |\xi| \to \infty.
\end{equation*}

Since $\cos \frac{\pi \gamma}{2 \alpha}>0$, the Mittag-Leffler function grows exponentially without sufficiently rapid oscillations to average out the growth.
\end{proof}

For $\gamma = 1$, the analytic properties of $K_t$ were studied in \cite[Lemma 3.1]{su2021dispersive}. If $\alpha < \gamma \leq 1$, then $E_{\alpha}(i^{-\gamma}|\xi|^\beta)$ decays as $|\xi|^{-\beta}$ as $|\xi| \rightarrow \infty$, and therefore the statement (and the proof) given in \cite{su2021dispersive} applies to $K_t$. If $\alpha = \gamma$, then $E_{\alpha}(i^{-\gamma}|\xi|^\beta)$ behaves as $e^{-i|\xi|^{\frac{\beta}{\alpha}}}$ as $|\xi| \rightarrow \infty$. Hence a modification of \cite[Lemma 3.1]{su2021dispersive} yields
\begin{lemma}\label{fund_asymp}
Let $0 < \beta \leq d$. Define $C(\gamma) =$
$\begin{cases}
    \frac{1}{\alpha},\ \alpha = \gamma < 1,\\
    0,\ \alpha < \gamma \leq 1.
\end{cases}$ If $\frac{d}{\beta} \notin \mathbb{Z}$, then there exist $W \in L^\infty(\mathbb{R}^d)$ and non-zero constants $C_j,\ 1 \leq j \leq \lfloor \frac{d}{\beta} \rfloor$ such that
\begin{equation*}
    K_t(x) = C(\gamma) \mathcal{F}^{-1}[e^{-it|\xi|^{\frac{\beta}{\alpha}}}] + |x|^{-d}\sum_{k=1}^{\lfloor \frac{d}{\beta} \rfloor} C_k \left(\frac{|x|^\beta}{t^\alpha}\right)^{k} + t^{-\frac{d\alpha}{\beta}} W\left(\frac{x}{t^{\frac{\alpha}{\beta}}}\right).
\end{equation*}
If $\frac{d}{\beta} = m \in \mathbb{Z}$, then there exist $W \in L^\infty(\mathbb{R}^d)$, non-zero constants $C_j,\ 1 \leq j \leq m$, and $W_1 \in L^\infty_{|x|\geq 1}$ where $W_1(x) \sim \ln |x|$ as $|x| \rightarrow 0$ such that
\begin{equation*}
    K_t(x) = C(\gamma) \mathcal{F}^{-1}[e^{-it|\xi|^{\frac{\beta}{\alpha}}}] + |x|^{-d}\sum_{k=1}^{m-1} C_k \left(\frac{|x|^\beta}{t^\alpha}\right)^{k} + C_m t^{-m\alpha}W_1\left(\frac{x}{t^{\frac{\alpha}{\beta}}}\right) +t^{-\frac{d\alpha}{\beta}} W\left(\frac{x}{t^{\frac{\alpha}{\beta}}}\right).
\end{equation*}
\end{lemma}

\begin{proof}[Proof of \Cref{ftw3}]
For $\alpha = 1$, \eqref{stfsch} is the half-wave equation whose solution is given by \eqref{half-wave2}. Assume $\alpha < 1$.

Let $\phi \in \mathcal{E}^\prime(\mathbb{R}^d),\ \psi \in \mathcal{S}(\mathbb{R}^d)$. Define $R_{t,\alpha} = K_t - \frac{1}{\alpha}K_t^{hw}$ where the explicit forms of $K_t$ are given by \Cref{fund_asymp}. For $0 < \rho < d$, it follows from linear distribution theory that
\begin{equation*}
\langle |x|^{-\rho} \ast \phi, \psi \rangle = \langle \phi, |x|^{-\rho} \ast \psi \rangle \in \mathbb{C}    
\end{equation*}
since $\mathcal{S}^\prime \ast \mathcal{S}$ embeds into the space of slowly increasing smooth functions. Therefore the $k=1$ term of the finite sum of $K_t$ in \Cref{fund_asymp} yields the dominant decay rate $t^{-\alpha}$.

If $\frac{d}{\beta} = m \in \mathbb{Z}$, then $m\alpha = d$, and for $\epsilon>0$ sufficiently small and $\Tilde{\phi}(x) := \phi(-x)$,
\begin{equation}\label{W}
\begin{aligned}
\left|\left< t^{-d} W_1(\frac{x}{t} )\ast \phi, \psi\right>\right| &\leq \int_{|x|\leq t} \left| t^{-d} W_1(\frac{x}{t}) \left(\Tilde{\phi} \ast \psi\right) \right|dx + \int_{|x| > t} \left| t^{-d} W_1(\frac{x}{t}) \left(\Tilde{\phi} \ast \psi \right) \right|dx\\
&\lesssim \frac{1}{t^{d-\epsilon}} + \frac{\| W_1 \|_{L^\infty}\| \Tilde{\phi} \ast \psi\|_{L^1}}{t^d} \lesssim t^{-\alpha},
\end{aligned}
\end{equation}
since $W_1(x) \sim \ln |x|$ as $|x| \rightarrow 0$ and $\mathcal{E}^\prime \ast \mathcal{S}$ embeds into $\mathcal{S}$. Similarly,
\begin{equation}\label{W1}
\left|\left< t^{-d} W(\frac{x}{t} )\ast \phi, \psi\right>\right| \lesssim t^{-\alpha}.    
\end{equation}    

To show \eqref{disp_est6}, observe that for $\rho = d - \beta k$, we have
\begin{equation*}
    \left| (|\cdot|^{-\rho} \ast \phi)(x)\right| \leq \sum_{N \in 2^{\mathbb{Z}}} \int |\xi|^{-\beta k} |\widehat{P_N \phi}(\xi)| d\xi \lesssim_d \sum_{N \in 2^{\mathbb{Z}}} N^{d-\beta k} \|P_N \phi \|_{L^1} \lesssim \| \phi\|_{B^{d-\beta}_{1,1}}.
\end{equation*}

Hence the Riesz potential term of $K_t$ in \Cref{fund_asymp} is bounded above by $t^{-\alpha} \| \phi\|_{B^{d-\beta}_{1,1}}$. The remaining estimates regarding $W,W_1$ follow as \eqref{W}, \eqref{W1}.

Since $E_{\alpha}(t)$ is a convolution operator with the multiplier $m_{t,\alpha}(\xi) := E_\alpha(i^{-\gamma}t^\alpha |\xi|^\alpha) - \frac{1}{\alpha}e^{-it|\xi|}$, $E_\alpha(t)$ is bounded on $L^2(\mathbb{R}^d)$ and
\begin{equation*}
    \| E_{\alpha}(t)\|_{L^2_x \rightarrow L^2_x} = \| m_{t,\alpha} \|_{L^\infty_{\xi}} = \frac{1-\alpha}{\alpha}\nrightarrow 0\ \text{as}\ t \to \infty,
\end{equation*}
where it can be shown by direct computation that $|m_{t,\alpha}(\xi)|$ is a continuous function in $\xi$ with the global maximum at $\xi = 0$ that monotonically decays as $|\xi|$ increases. 

To show strong convergence, let $\phi \in L^2(\mathbb{R}^d) \setminus \{0\}$. For $|t| \geq 1,\ 0 < \epsilon < \| \phi\|_{L^2}$, define
\begin{equation*}
R_t \simeq \frac{\| \phi \|_{L^2}^{1/\alpha}}{|t| \epsilon^{1/\alpha}}.    
\end{equation*}
 Then
 \begin{equation*}
     \| E_{\alpha}(t)\phi\|_{L^2}^2 = \int_{|\xi|\leq R_t} |m_{t,\alpha}(\xi)|^2 |\widehat{\phi}(\xi)|^2 d\xi+ \int_{|\xi| > R_t} |m_{t,\alpha}(\xi)|^2 |\widehat{\phi}(\xi)|^2 d\xi := I + II.
 \end{equation*}
By \eqref{mlf_decay4},

\begin{equation*}
    II \lesssim \left(\frac{\| \phi \|_{L^2}}{|t R_t|^\alpha}\right)^2 \lesssim \epsilon^2,
\end{equation*}
and
\begin{equation*}
    I \leq \left(\frac{1-\alpha}{\alpha}\right)^2 \cdot \|\widehat{\phi} \|_{L^2_{|\xi| \leq R_t}}^2 \leq \epsilon^2,
\end{equation*}
for $t>0$ sufficiently large.
\end{proof}

Sharp frequency-localized dispersive estimates reveal an interesting role of the fractional time derivative. An analytic result is given in \Cref{disp_est} and a numerical evidence, in \Cref{fig}.

\begin{proposition}\label{disp_est}
Let $N \in 2^\mathbb{Z}$. For $\alpha = \gamma \in (0,1)$ and $\alpha \neq \beta$,
\begin{equation}\label{disp_est2}
    \| P_N K_t \ast f\|_{L^\infty(\mathbb{R}^d)} \lesssim N^d \left(\frac{1}{1+t^\alpha N^\beta} + \frac{1}{1+t^\frac{d}{2} N^\frac{d\beta}{2\alpha}}\right) \| f \|_{L^1}.
\end{equation}
For $\alpha < \gamma \leq 1$,
\begin{equation}\label{disp_est3}
    \| P_N K_t \ast f\|_{L^\infty(\mathbb{R}^d)} \lesssim  \frac{N^d}{1+t^\alpha N^\beta} \| f \|_{L^1}.
\end{equation}
\end{proposition}

\begin{remark}\label{sharp}
    The frequency-localized dispersive estimates given by \eqref{disp_est2}, \eqref{disp_est3} are sharp by combining \cite[Proposition 1.3]{su2021dispersive} and \cite[Proposition 2]{cho2011remarks}. More precisely, there exist $N_0 \in 2^{\mathbb{Z}}$ and $t_0>0$ such that for all $N > N_0,\ t > t_0$, we have    
\begin{equation*}
    \| P_N K_t\|_{L^\infty(\mathbb{R}^d)} \gtrsim N^d \left(\frac{1}{1+t^\alpha N^\beta} + \frac{1}{1+t^\frac{d}{2} N^\frac{d\beta}{2\alpha}}\right),
\end{equation*}
which shows the sharpness of \eqref{disp_est2}, and similarly for \eqref{disp_est3}.

For $d \geq 2$, the RHS of \eqref{disp_est2} decays as $t^{-\alpha}$, independent of the spatial dimension. For $d=1$, the dominant time decay is $t^{-\frac{1}{2}}$ if $\alpha \in (\frac{1}{2},1)$, and $t^{-\alpha}$ if $\alpha \in (0,\frac{1}{2}]$. Numerical evidence of the last observation is given in \Cref{fig}.
\end{remark}

\begin{figure}[htbp]
\centering
\includegraphics[width = 0.78\textwidth]{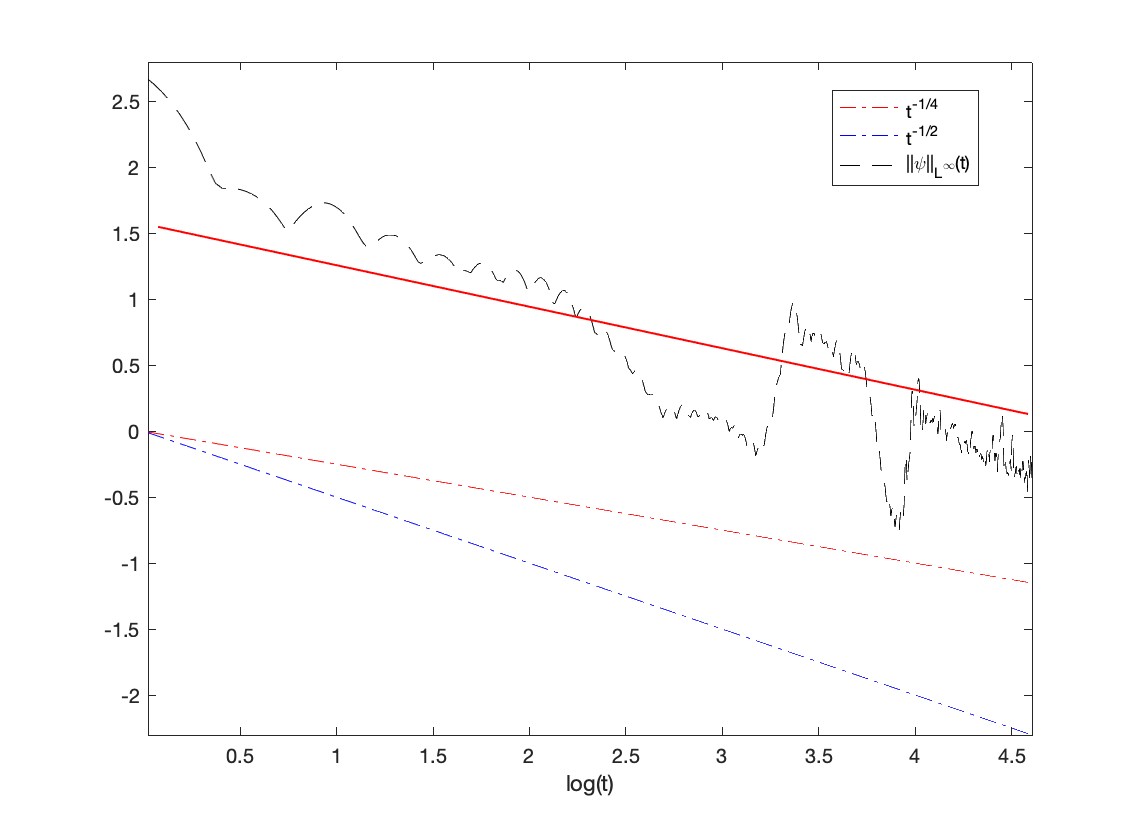}
\caption{Log-log plot of $\|\psi(t)\|_{L^\infty} = \||u(t)|^2\|_{L^\infty}$ where $u(t)$ is the solution to \eqref{stfsch} with $\alpha = \gamma =\frac{1}{4},\ \beta = \frac{1}{2}$. The initial datum is $u(0) = \frac{\sin(2x)-\sin(x)}{x}$ whose Fourier transform is a non-zero constant around an annulus in the Fourier space, and zero otherwise. The solution was computed in Matlab using the fundamental solution provided in \Cref{cor41} by utilizing the Mittag-Leffler file provided by Podlubny (Igor Podlubny (2023). Mittag-Leffler function, https://www.mathworks.com/matlabcentral/fileexchange/8738-mittag-leffler-function). For comparison, plots of $t^{-1/2}, t^{-1/4}$, and $\| \psi(t) \|_{L^\infty}$ are given where the solid line is provided for the reader's convenience in observing the average decay of $\psi$. The jumps in the decay are caused by constructive and destructive resonances, but note that the decay of $\| \psi(t) \|_{L^\infty}$ is consistent with the theoretical rate of $t^{-1/4}$ given in  \Cref{disp_est}. }\label{fig}
\end{figure}

\begin{proof}[Proof of \Cref{disp_est}]

Let $\alpha = \gamma \in (0,1)$ and $\alpha \neq \beta$.
\begin{equation*}
P_N K_t (x) = \int E_\alpha(i^{-\gamma} t^\alpha |\xi|^\beta) \zeta\left(\frac{|\xi|}{N}\right)e^{i x \cdot \xi} d\xi
=t^{-d\frac{\alpha}{\beta}} \int E_\alpha(i^{-\gamma}|\xi|^\beta) \zeta\left(\frac{|\xi|}{N_1}\right)d\xi,
\end{equation*}
where $N_1 = t^{\frac{\alpha}{\beta}}N$. By \eqref{mlf_decay4},
\begin{equation*}
    \int E_{\alpha}(i^{-\gamma} |\xi|^\beta) \zeta\left(\frac{|\xi|}{N_1}\right)d\xi = \frac{1}{\alpha} \int e^{-i|\xi|^{\frac{\beta}{\alpha}}}\zeta\left(\frac{|\xi|}{N_1}\right)d\xi + \int R(\xi)\zeta\left(\frac{|\xi|}{N_1}\right)d\xi,
\end{equation*}
where $R = O(|\xi|^{-\beta})$ as $|\xi|\rightarrow \infty$. By \cite[Proposition 1]{cho2011remarks} and \cite[Theorem 1.2]{su2021dispersive}, respectively,
\begin{equation*}
\begin{aligned}
\left|\int e^{-i|\xi|^{\frac{\beta}{\alpha}}}\zeta\left(\frac{|\xi|}{N_1}\right)d\xi\right| &\lesssim \frac{N_1^d}{1+N_1^{\frac{d\beta}{2\alpha}}}\\
\left|\int R(\xi)\zeta\left(\frac{|\xi|}{N_1}\right)d\xi\right| &\lesssim \frac{N_1^d}{1+N_1^\beta}.
\end{aligned}    
\end{equation*}

Altogether we have
\begin{equation*}
\| P_N K_t\|_{L^\infty(\mathbb{R}^d)} \lesssim N^d \left(\frac{1}{1+t^\alpha N^\beta} + \frac{1}{1+t^\frac{d}{2} N^\frac{d\beta}{2\alpha}}\right),    
\end{equation*}
and hence \eqref{disp_est2} by the Young's inequality. The estimate \eqref{disp_est3} follows similarly.
\end{proof}

\section{Conclusion}

In this paper, we established infinite speed of propagation (ISP) for a broad class of nonlinear dispersive equations, showing that compactly supported solutions cannot persist over time under minimal regularity assumptions on the dispersion relation. Using complex-analytic tools, including the Paley–Wiener–Schwartz theorem and the FBI transform, we extended known results beyond polynomial dispersion relations to fractional and more general settings. We also highlighted a key limitation of Bourgain’s complex-analytic approach, which requires cubic growth of the dispersion relation (\Cref{rmk:3.1}) and thus excludes cases like the nonlinear Schrödinger equation. For fractional dispersive systems, we demonstrated how dispersive decay rates are shaped by the interplay between memory effects and spatial scaling, revealing new qualitative phenomena. These results contribute to the broader theory of dispersive and nonlocal PDEs, advancing our understanding of support propagation, unique continuation, and analytic regularity. Future work may address ISP in nonlinear and stochastic models and explore the influence of boundary conditions and external forcing.

\section{Declarations.}
\subsection{Funding and/or Conflicts of interests/Competing interests}

Both authors are supported by NSF RTG grant DMS-1840260. There are no conflicting interests.

%Since $E_\alpha(z)$ is of order $\frac{1}{\alpha}$ (see \cite[Chapter 1]{podlubny1999fractional}), $K_t(\xi)$ grows as $e^{|\xi|^{\frac{\beta}{\alpha}}}$ for $|\xi| \to \infty$, and by \Cref{PWS}, the condition $\alpha \geq \beta$ is necessary for the solution support to be compact for $t\neq 0$. Changing the role of space and time, the condition $\alpha \leq \beta$ is obtained, and hence $\alpha = \beta$.

%\begin{remark}
%On the other hand, the condition $\gamma \leq 1$ is not necessary for \Cref{ftw3} to hold, although a restriction on $\gamma$ is certainly needed. For example, consider $\alpha = \beta = 1,\ \gamma =2$, i.e.
%\begin{equation*}
    %\partial_t u = - (-%\Delta)^{\frac{1}{2}}u,
%\end{equation*}
%whose fundamental solution given is by
%\begin{equation*}
%    K_t = \frac{1}{\pi t} \frac{1}{1+\left(\frac{x}{t} \right)^2},
%\end{equation*}
%cannot be considered as a traveling wave.
%\end{remark}

\bibliographystyle{abbrv}
\bibliography{ref}
\end{document}